\documentclass[11pt,a4paper]{article}
\usepackage[T1]{fontenc}
\usepackage{lmodern,amsmath,amssymb,amsthm,mathtools}
\usepackage[margin=20mm]{geometry}
\usepackage{enumitem,microtype}
\usepackage[authoryear,round]{natbib}
\usepackage[colorlinks=true,linkcolor=blue,citecolor=blue,urlcolor=blue]{hyperref}

\newtheorem{theorem}{Theorem}[section]
\newtheorem{lemma}[theorem]{Lemma}
\newcommand{\E}{\mathbb E}
\newcommand{\R}{\mathbb R}
\newcommand{\F}{\mathcal F}
\newcommand{\Lip}{\mathrm{Lip}}

\newcommand{\norm}[1]{\left\lVert #1\right\rVert}
\newcommand{\Norm}[1]{\Big\lVert #1\Big\rVert}

\newcommand{\shift}{\Delta^{\mathrm{shift}}_{[0,1]}}
\newcommand{\rate}{R^n_{H,\varepsilon}(s,t)}

\newcommand{\A}{\mathcal A}
\newcommand{\dd}{\,\mathrm d}
\newcommand{\kn}{\kappa_n}
\newcommand{\ph}{\varphi}

\setlist[enumerate]{label=(\roman*),leftmargin=*,itemsep=2pt,topsep=3pt}
\allowdisplaybreaks[2]
\date{}
\title{Strong convergence of Euler--Maruyama schemes for additive fractional SDEs with Lipschitz drift}
\author{Tsukasa Moritoki\\[3pt]
\small Department of Mathematics, Okayama University,\\
\small Tsushima-naka, Kita-ku Okayama 700-8530, Japan\\
\small \texttt{pxyq0l31@s.okayama-u.ac.jp}}
\begin{document}\maketitle\begin{abstract}
We study the strong convergence rate of the Euler--Maruyama scheme for additive stochastic differential equations driven by a fractional Brownian motion with Hurst parameter $H\in(0,1)$. Assuming the drift coefficient to be Lipschitz continuous, we establish strong convergence bounds of order $1$ if $H\in(1/2,1)$, and $1/2+H-\varepsilon$, for any $\varepsilon>0$, if $H\in(0,1/2]$. The main ingredient is a shifted stochastic sewing argument, which exploits the conditional Gaussian structure of fractional Brownian motion to control the noise discretization error.
\end{abstract}
\noindent\textbf{Keywords:} Euler--Maruyama scheme; fractional Brownian motion; strong convergence rate; stochastic sewing lemma; Lipschitz drift.
\section{Introduction}
We consider the additive fractional stochastic differential equation
\[
X_t=x_0+\int_0^t b(X_s)\dd s+B_t^H,\qquad t\in[0,1],
\]
where $x_0\in\R^d$ and $b:\R^d\to\R^d$ is Lipschitz continuous. Here $B^H$ is a $d$-dimensional fractional Brownian motion (fBm) with Hurst parameter $H\in(0,1)$. The Euler--Maruyama scheme is defined by
\[
X_t^n=x_0+\int_0^t b(X^n_{\kn(s)})\dd s+B_t^H,\qquad
\kn(t):=\lfloor nt\rfloor/n,\quad t\in[0,1].
\]
We work on $[0,1]$ for notational simplicity; the same estimates hold on any finite interval $[0,T]$, with constants depending additionally on $T$.
The main result of this paper is the following.
\begin{theorem}\label{thm:main}
Let $p\ge1$. Assume that the drift coefficient $b$ is Lipschitz continuous. Then the following estimates hold. If $H\in(1/2,1)$, there exists a constant
$C=C(p,d,H,|x_0|,|b(0)|,[b]_{\Lip})>0$ such that
\[
\Norm{\sup_{t\in[0,1]}|X_t-X_t^n|}_{L^p}\le Cn^{-1}.
\]
If $H\in(0,1/2]$, then for any $\varepsilon>0$ there exists a constant
$C=C(p,d,H,\varepsilon,|x_0|,|b(0)|,[b]_{\Lip})>0$ such that
\[
\Norm{\sup_{t\in[0,1]}|X_t-X_t^n|}_{L^p}\le Cn^{-(1/2+H)+\varepsilon}.
\]
\end{theorem}
Strong convergence rates for numerical schemes of additive fractional SDEs have been studied under various assumptions on the drift. For bounded $b\in C^\alpha$, \citet{BDG21} proved an upper convergence order $(1/2+\alpha H)\wedge1-\varepsilon$. In particular, for bounded Lipschitz drifts, this gives the order $1/2+H-\varepsilon$ when $H\le1/2$ and $1-\varepsilon$ when $H>1/2$. For $H\in(1/3,1/2)$, \citet{HW23} obtained the order $2H$ for general nonlinear drifts under higher smoothness assumptions, and the order $H+1/2$ when the drift is linear. For $H>1/2$, numerical approximation of additive fractional SDEs was also studied by \citet{N06}; the setting there differs from ours in the information model, error criterion, dimension, and regularity assumptions.

\citet{N06} obtained the exact rate $H$ for scalar SDEs with additive fractional noise for $H>1/2$, under stronger smoothness assumptions on the coefficients. His result concerns the $L^2(\Omega;L^2([0,1]))$-error and, in the optimality analysis, approximation methods based on finitely many bounded linear functionals of the driving fractional Brownian motion. In contrast, we consider a multidimensional continuous-time Euler--Maruyama scheme, measure the error in $L^p(\Omega;C([0,1];\R^d))$, and assume only that the drift is globally Lipschitz, without boundedness or higher differentiability assumptions.

The contribution of the present paper is twofold. First, we extend the bounded Lipschitz result of \citet{BDG21} to globally Lipschitz drifts that may be unbounded. Second, for $H>1/2$, we improve the previously available upper convergence order $1-\varepsilon$ to $1$, thereby removing the $\varepsilon$-loss. For $H\le 1/2$, we obtain the same upper convergence order as \citet{BDG21}, now without the boundedness assumption on the drift.

The main technical point is the estimate of the noise discretization error
\[
\int_s^t\{b(B_r^H+\ph^n_{\kn(r)})-b(B_{\kn(r)}^H+\ph^n_{\kn(r)})\}\dd r,
\qquad \ph_t^n:=X_t^n-B_t^H.
\]
The main difficulty is that, in the present argument, one needs a conditioning time separated from the relevant integration interval in order to exploit the conditional Gaussian structure of fractional Brownian motion. We therefore use a shifted stochastic sewing lemma, in which the conditioning time is moved to the past. This keeps the time parameter in the conditional heat semigroup under control and makes it possible to combine conditional Gaussian estimates with heat-kernel bounds.

Section 2 gives the notation and states the preliminary lemmas.
Section 3 states the key noise discretization estimate and provides a proof outline.
Section 4 gives the proof of Theorem \ref{thm:main} in both Hurst regimes.
Proofs of Lemmas 2.2--2.4 and the complete proof of Lemma \ref{lem:key} are given in the Supplementary Material.
Lemma numbers are the same in both documents.

\section{Preliminaries}\subsection{Notation and fractional Brownian motion}
Throughout, we set $\ph_t:=X_t-B_t^H$ and $\ph_t^n:=X_t^n-B_t^H$.
For $q\ge1$, write $L^q:=L^q(\Omega)$. For a stochastic process $f$, an interval $I\subset[0,1]$, and $\tau\in(0,1]$, set
\[
[f]_{C^\tau L^q(I)}:=\sup_{u,v\in I,\,u\ne v}\frac{\norm{f_v-f_u}_{L^q}}{|v-u|^\tau},
\qquad [f]_{C^0L^q(I)}:=\sup_{u\in I}\norm{f_u}_{L^q}.
\]
Denote by $[g]_{\Lip}$ the Lipschitz constant of $g$. For $t>0$, let
\[
p_t(y):=(2\pi t)^{-d/2}\exp\bigl(-|y|^2/(2t)\bigr),\qquad P_tf:=p_t*f.
\]
We use the convention $P_0f=f$. Set
\[
\Delta_I:=\{(s,t)\in I^2:s<t\},\qquad
\shift:=\{(s,t)\in[0,1]^2:s<t,\ s-(t-s)\ge0\}.
\]
We write $A\lesssim B$ if $A\le LB$ for a positive constant $L$. Implicit constants may depend on the fixed parameters in the corresponding statement, but are independent of $n$ and the time variables.

Consider a probability space $(\Omega,\F,\mathbb P)$ carrying a $d$-dimensional two-sided Brownian motion $W=(W_t)_{t\in\R}$. 
Let $(\F_t)_{t\in\R}$ be the completed filtration generated by its increments, and write $\E^s[\cdot]=\E[\cdot\mid\F_s]$. We use the Mandelbrot--Van Ness representation
\[
B_t^H=C_H\left\{\int_{-\infty}^0\bigl(|t-a|^{H-1/2}-|a|^{H-1/2}\bigr)\dd W_a
+\int_0^t|t-a|^{H-1/2}\dd W_a\right\},
\quad t\in[0,1],
\]
where $C_H>0$ is chosen so that $\E[|B_t^H|^2]=dt^{2H}$; see \citet[Chapter 5, Proposition 5.1.2]{Nualart06}. The following facts follow from \citet[Proposition 3.6]{BDG21}.
\begin{lemma}\label{lem:fbm}
Let $q\ge1$ and $H\in(0,1)$. For $0\le s\le t\le1$, set
$c(s,t):=C_H|t-s|^H/\sqrt{2H}$. Then:
\begin{enumerate}
\item There exists $C=C(q,d,H)>0$ such that for all $0\le s\le t\le1$,
\[\norm{B_t^H-B_s^H}_{L^q}=C|t-s|^H.\]
\item For $0\le s\le t\le1$, any $\F_s$-measurable $\R^d$-valued random variable $Z$, and any Borel measurable $f:\R^d\to\R$ satisfying $\E[|f(B_t^H+Z)|]<\infty$,
\[\E^s[f(B_t^H+Z)]=P_{c^2(s,t)}f\bigl(\E^s[B_t^H]+Z\bigr).\]
\item There exists $C=C(q,d,H)>0$ such that for $0\le s\le u\le t\le1$ with $t-u\le u-s$,
\[\norm{\E^s[B_t^H]-\E^s[B_u^H]}_{L^q}\le C|t-u||t-s|^{H-1}.\]
\end{enumerate}
\end{lemma}

\subsection{Technical lemmas}

We collect the remaining estimates used in the proof of the main theorem.

We use the following heat-kernel estimates; similar estimates appear in \citet[Lemmas 3.1 and 5.1]{B25}.
\begin{lemma}\label{lem:heat}
Let $f:\R^d\to\R$ be Lipschitz continuous. There exists $C=C(d)>0$ such that, for $0<s\le t\le1$,
\begin{enumerate}
\item $[P_tf]_{\Lip}\le C[f]_{\Lip}$;
\item $[\nabla P_tf]_{\Lip}\le C[f]_{\Lip}t^{-1/2}$;
\item $\norm{P_tf-P_sf}_{\infty}\le C[f]_{\Lip}s^{-1/2}(t-s)$;
\item $[P_tf-P_sf]_{\Lip}\le C[f]_{\Lip}s^{-1}(t-s)$.
\end{enumerate}
The estimates are also valid componentwise for vector-valued functions, with the constant depending only on $d$.
\end{lemma}
\begin{lemma}\label{lem:mom}
Let $q\ge1$. There exists $C=C(q,d,H,|x_0|,|b(0)|,[b]_{\Lip})>0$ such that,
\[
\sup_{n \in \mathbb{N}} \sup_{t\in[0,1]}\norm{X_t^n}_{L^q}\le C,\qquad
\sup_{n \in \mathbb{N}} \norm{X_v^n-X_u^n}_{L^q}\le C|v-u|^H,\quad 0\le u\le v\le1.
\]
\end{lemma}
The next estimate plays the same role as \citet[Lemma 5.20]{B25} in the present setting, where the drift is not assumed to be bounded.
\begin{lemma}\label{lem:fluct}
Assume that $b:\R^d\to\R^d$ is Lipschitz continuous. For all $q\ge2$, $n\in\mathbb N$, and $(s,t)\in\Delta_{[0,1]}$, there exists
$C=C(q,d,H,|x_0|,|b(0)|,[b]_{\Lip})>0$ such that
\[
\norm{\ph^n_{\kn(t)}-\E^s[\ph^n_{\kn(t)}]}_{L^q}\le C(t-s)^{1+H}.
\]
\end{lemma}
The stochastic sewing lemma was introduced by \citet{Le20}. We use the shifted version introduced by \citet[Lemma 2.2]{G23}; see also \citet[Lemma 3.1]{BDG25} for the formulation below. Write $\delta A_{s,u,t}:=A_{s,t}-A_{s,u}-A_{u,t}$.
\begin{lemma}\label{lem:sew}
Let $q\ge2$. Suppose that there exist measurable functions $\A:\Omega\times[0,1]\to\R^d$, $A:\Omega\times\shift\to\R^d$, such that:
\begin{enumerate}
\item For any $(s,t)\in\Delta_{[0,1]}$,
\[
\A_t-\A_s=\lim_{m\to\infty}\sum_{i=1}^{m-1} A_{s+i(t-s)/m,\,s+(i+1)(t-s)/m}
\quad\text{in probability}.
\]
\item For any $(s,t)\in\shift$, $A_{s,t}$ is $\F_t$-measurable.
\item There exist $\Gamma_1,\Gamma_2,\varepsilon_1,\varepsilon_2>0$ such that, for every $(s,t)\in\shift$ and $u=(s+t)/2$,
\begin{align}
\norm{A_{s,t}}_{L^q}&\le\Gamma_1|t-s|^{1/2+\varepsilon_1},\label{eq:sew1}\\
\norm{\E^{s-(t-s)}\delta A_{s,u,t}}_{L^q}&\le\Gamma_2|t-s|^{1+\varepsilon_2}.\label{eq:sew2}
\end{align}
\end{enumerate}
Then there exist $K_1,K_2>0$, depending only on $\varepsilon_1,\varepsilon_2,q,d$, such that, for $0\le s\le t\le1$,
\[
\norm{\A_t-\A_s}_{L^q}\le K_1\Gamma_1|t-s|^{1/2+\varepsilon_1}
+K_2\Gamma_2|t-s|^{1+\varepsilon_2}.
\]
\end{lemma}
We also use the local-to-global estimate of \citet[Lemma 4.8]{B25}.
\begin{lemma}\label{lem:global}
Let $q\ge1$, $\tau\in(0,1]$, $\Gamma\ge0$, $\ell>0$, and $M>0$. Let $f:\Omega\times[0,1]\to\R^d$ be measurable. Suppose that $f_0=0$ and that, for every $(s,t)\in\Delta_{[0,1]}$ with $t-s\le\ell$,
\[
[f]_{C^\tau L^q([s,t])}\le M\norm{f_s}_{L^q}+\Gamma.
\]
Then there exists $C=C(\tau,\ell,M)>0$, independent of $\Gamma$, such that
$[f]_{C^\tau L^q([0,1])}\le C\Gamma$.
\end{lemma}

\section{Noise discretization estimate}
The proof is inspired by \citet[Lemma 4.1]{BDG21}, \citet[Lemma 4.7]{BDG25}, and \citet[Lemma 4.2]{SLW25}.
\begin{lemma}\label{lem:key}
Let $p\ge2$, $n\in\mathbb N$, and $(s,t)\in\Delta_{[0,1]}$. Suppose that $b$ is Lipschitz continuous. If $H\in(1/2,1)$, then there exists
$C=C(p,d,H,|x_0|,|b(0)|,[b]_{\Lip})>0$ such that
\[
\Norm{\int_s^t\{b(B_r^H+\ph^n_{\kn(r)})-b(B_{\kn(r)}^H+\ph^n_{\kn(r)})\}\dd r}_{L^p}
\le Cn^{-1}(t-s)^H.
\]
If $H\in(0,1/2]$, then for any $\varepsilon\in(0,1/2)$ there exists
$C=C(p,d,H,|x_0|,|b(0)|,[b]_{\Lip},\varepsilon)>0$ such that
\[
\Norm{\int_s^t\{b(B_r^H+\ph^n_{\kn(r)})-b(B_{\kn(r)}^H+\ph^n_{\kn(r)})\}\dd r}_{L^p}
\le Cn^{-(1/2+H)+\varepsilon}(t-s)^{1/2+\varepsilon}.
\]
\end{lemma}

\begin{proof}[Proof outline]
Define
\[
\mathcal A_t
:=
\int_0^t
\Bigl\{
b(B_r^H+\varphi_{\kappa_n(r)}^n)
-
b(B_{\kappa_n(r)}^H+\varphi_{\kappa_n(r)}^n)
\Bigr\}\,\mathrm{d}r .
\]
We apply Lemma \ref{lem:sew} to estimate the increment
$\mathcal A_t-\mathcal A_s$.

For $(s,t)\in\Delta_{[0,1]}^{\rm shift}$, set
\[
s_1:=s-(t-s),
\qquad
\psi_r:=\varphi_{\kappa_n(r)}^n,
\]
and define
\[
A_{s,t}
:=
\E^{s_1}
\left[
\int_s^t
\Bigl\{
b(B_r^H+\E^{s_1}\psi_r)
-
b(B_{\kappa_n(r)}^H+\E^{s_1}\psi_r)
\Bigr\}\,\mathrm{d}r
\right].
\]
The shift of the conditioning time to $s_1<s$ allows us to exploit the
conditional Gaussian structure of $B^H$ throughout the interval $[s,t]$.

We first estimate the germ $A_{s,t}$. By Lemma \ref{lem:fbm}(ii), the conditional
expectations can be represented by the heat semigroup. Combining this
representation with the spatial and time regularity estimates in
Lemma \ref{lem:heat}, while treating the part of length $O(n^{-1})$ near the left
endpoint directly, yields
\[
\|A_{s,t}\|_{L^p}
\lesssim
n^{-1}(t-s)^H,
\qquad H\in(1/2,1),
\]
and
\[
\|A_{s,t}\|_{L^p}
\lesssim
n^{-(1/2+H)+\varepsilon}(t-s)^{1/2+\varepsilon},
\qquad H\in(0,1/2].
\]

Next, let $u=(s+t)/2$. To estimate $\norm{\E^{s - (t - s)}[\delta A_{s,u,t}]}_{L^p}$, we compare conditional expectations corresponding to
different conditioning times. The key input is Lemma \ref{lem:fluct}, which gives
\[
\|\E^{s_1}[\psi_r]-\E^{s_2}[\psi_r]\|_{L^q}
\lesssim
(r-s_1)^{1+H},
\qquad
s_2:=s-(u-s).
\]
Using this estimate together with Lemmas \ref{lem:fbm} and \ref{lem:heat}, we obtain
\[
\left\|
\E^{s-(t-s)}[\delta A_{s,u,t}]
\right\|_{L^p}
\lesssim
n^{-1}(t-s)^{1+H}.
\]
The only product term that requires additional care is controlled by
H\"older's inequality, the $L^p$-contraction property of conditional
expectation, the conditional increment estimate in Lemma \ref{lem:fbm}(iii), and
Lemma \ref{lem:fluct}.

Finally, we verify Lemma \ref{lem:sew}(i). For the uniform partition
\[
t_i=s+i(t-s)/m,
\]
we show that
\[
\sum_{i=1}^{m-1}A_{t_i,t_{i+1}}
\longrightarrow
\mathcal A_t-\mathcal A_s
\]
in probability. The difference is decomposed into a boundary term,
two martingale-difference terms, and a remainder. These terms are controlled,
respectively, by the Lipschitz continuity of $b$, the
Burkholder--Davis--Gundy inequality, and Lemma \ref{lem:fluct}.

Applying Lemma \ref{lem:sew} therefore gives
\[
\|\mathcal A_t-\mathcal A_s\|_{L^p}
\lesssim
n^{-1}(t-s)^H,
\qquad H\in(1/2,1),
\]
and
\[
\|\mathcal A_t-\mathcal A_s\|_{L^p}
\lesssim
n^{-(1/2+H)+\varepsilon}(t-s)^{1/2+\varepsilon},
\qquad H\in(0,1/2].
\]
This proves the claim. Full details are given in the Supplementary
Material.
\end{proof}

\section{Proof of the main theorem}
We use a local-to-global argument, based on Lemma \ref{lem:global}, following the proof strategy of \citet[Theorem 1.5]{B25}.
\begin{proof}[Proof of Theorem~\ref{thm:main}]
It suffices to prove the theorem for $p>4$. Indeed, for a prescribed $p\le4$, choose $q>\max\{p,4\}$, prove the estimate in $L^q$, and use $\norm{Y}_{L^p}\le\norm{Y}_{L^q}$.
Write $e_t:=\ph_t-\ph_t^n=X_t-X_t^n$. Fix $(s,t)\in\Delta_{[0,1]}$ and $(u,v)\in\Delta_{[s,t]}$. The drift difference gives
\begin{align*}
\norm{e_v-e_u}_{L^p}
&\le\Norm{\int_u^v\{b(B_r^H+\ph_r)-b(B_r^H+\ph_r^n)\}\dd r}_{L^p}\\
&\quad+\Norm{\int_u^v\{b(B_r^H+\ph_r^n)-b(B_r^H+\ph_{\kn(r)}^n)\}\dd r}_{L^p}\\
&\quad+\Norm{\int_u^v\{b(B_r^H+\ph_{\kn(r)}^n)-b(B_{\kn(r)}^H+\ph_{\kn(r)}^n)\}\dd r}_{L^p}\\
&=:J_1+J_2+J_3.
\end{align*}
By Lipschitz continuity and Lemma~\ref{lem:mom},
\begin{align*}
J_1&\le[b]_{\Lip}[e]_{C^0L^p([u,v])}(v-u),\\
J_2&\le[b]_{\Lip}(v-u)\sup_{r\in[u,v]}
\Norm{\int_{\kn(r)}^r b(X_{\kn(a)}^n)\dd a}_{L^p}
\lesssim n^{-1}(v-u).
\end{align*}
Here linear growth and Lemma~\ref{lem:mom} give $\sup_{n,a}\norm{b(X^n_{\kn(a)})}_{L^p}<\infty$.

Suppose first that $H\in(1/2,1)$. Lemma~\ref{lem:key} gives
\[
\norm{e_v-e_u}_{L^p}\le C_1[e]_{C^0L^p([u,v])}(v-u)
+C_2n^{-1}(v-u)+C_3n^{-1}(v-u)^H.
\]
Dividing both sides by $(v-u)^{1/2}$ and taking the supremum, we obtain
\begin{align*}
[e]_{C^{1/2}L^p([s,t])}
&\le C_1(t-s)^{1/2}[e]_{C^0L^p([s,t])}
+C_2n^{-1}(t-s)^{1/2}+C_3n^{-1}(t-s)^{H-1/2}\\
&\le C_1(t-s)^{1/2}\bigl(\norm{e_s}_{L^p}
+(t-s)^{1/2}[e]_{C^{1/2}L^p([s,t])}\bigr)\\
&\quad+C_2n^{-1}(t-s)^{1/2}+C_3n^{-1}(t-s)^{H-1/2}.
\end{align*}
The seminorms above are finite: the drift integrals have uniformly bounded $L^p$ integrands, using the same linear-growth and Gronwall estimate for $X$ as for $X^n$.
Choose $\ell>0$ with $C_1\ell<1/2$. Absorbing the seminorm term, for $t-s\le\ell$ we get
\[
[e]_{C^{1/2}L^p([s,t])}\lesssim\norm{e_s}_{L^p}+n^{-1}.
\]
Since $e_0=0$, Lemma~\ref{lem:global} gives $[e]_{C^{1/2}L^p([0,1])}\lesssim n^{-1}$.
As $p>4$, $1/4\in(0,1/2-1/p)$. The Kolmogorov continuity theorem
\citep[Theorem A.11]{FV10} therefore yields
\[
\norm{[X-X^n]_{C^{1/4}([0,1];\R^d)}}_{L^p}
=\norm{[e]_{C^{1/4}([0,1];\R^d)}}_{L^p}\le C_0n^{-1}.
\]
Here $[\cdot]_{C^{1/4}}$ is the usual pathwise H\"older seminorm. Since $e_0=0$,
\[
\Norm{\sup_{t\in[0,1]}|X_t-X_t^n|}_{L^p}\le C_0n^{-1}.
\]

Next suppose $H\in(0,1/2]$, and fix $0<\varepsilon_0<\min\{\varepsilon,1/2\}$. The second estimate in Lemma~\ref{lem:key} gives
\begin{align*}
\norm{e_v-e_u}_{L^p}
&\le C_1[e]_{C^0L^p([u,v])}(v-u)+C_2n^{-1}(v-u)\\
&\quad+C_3n^{-(1/2+H)+\varepsilon_0}(v-u)^{1/2+\varepsilon_0}.
\end{align*}
Dividing both sides by $(v-u)^{1/2}$ and taking the supremum, we obtain
\begin{align*}
[e]_{C^{1/2}L^p([s,t])}
&\le C_1(t-s)^{1/2}\bigl(\norm{e_s}_{L^p}
+(t-s)^{1/2}[e]_{C^{1/2}L^p([s,t])}\bigr)\\
&\quad+C_2n^{-1}(t-s)^{1/2}
+C_3n^{-(1/2+H)+\varepsilon_0}(t-s)^{\varepsilon_0}.
\end{align*}
For $t-s\le\ell$, absorb as above and use $n^{-1}\le n^{-(1/2+H)+\varepsilon_0}$ to obtain
\[
[e]_{C^{1/2}L^p([s,t])}\le C\norm{e_s}_{L^p}+Cn^{-(1/2+H)+\varepsilon_0}.
\]
Lemma~\ref{lem:global} then gives
$[e]_{C^{1/2}L^p([0,1])}\lesssim n^{-(1/2+H)+\varepsilon_0}$.
Since $p>4$, the same Kolmogorov continuity theorem and $e_0=0$ yield
\[
\Norm{\sup_{t\in[0,1]}|X_t-X_t^n|}_{L^p}
\le\norm{[e]_{C^{1/4}([0,1];\R^d)}}_{L^p}
\lesssim n^{-(1/2+H)+\varepsilon_0}
\le n^{-(1/2+H)+\varepsilon}.
\]
This proves both assertions, with the parameter dependence stated in the theorem.
\end{proof}
\section*{Declaration of generative AI and AI-assisted technologies in the writing process}
The author used ChatGPT for translation, proofreading, and language editing. The author reviewed and edited the content and takes full responsibility for the article.

\clearpage
\setcounter{page}{1}

\begin{center}
\vspace{1em}
{\LARGE Supplementary material to\par}
\vspace{0.4em}
{\LARGE
``Strong convergence of Euler--Maruyama schemes for additive fractional SDEs
with Lipschitz drift''\par}

\vspace{1em}

Tsukasa Moritoki

Department of Mathematics, Okayama University,\\
Tsushima-naka, Kita-ku, Okayama 700-8530, Japan
\end{center}

\setcounter{section}{0}
\setcounter{theorem}{0}

\setcounter{equation}{2}

\vspace{1em}

\begin{abstract}
This supplementary material contains the proofs of Lemmas 2.2, 2.3, 2.4, and 3.1 of the accompanying article. These comprise the heat-semigroup estimates, the uniform moment and increment estimates for the Euler--Maruyama approximation, the conditional fluctuation estimate, and the noise discretization estimate obtained by the shifted stochastic sewing lemma. The notation and numbering of lemmas and equations are those of the main article; in particular, references to (1) and (2) refer to the estimates in its Lemma 2.5.
\end{abstract}

\section*{Proof of Lemma 2.2}
\begin{proof}
    The estimates in (i), (ii) and (iv) follow by differentiating the Gaussian kernel, using the cancellations
    \[
    \int_{\R^d}\nabla p_t(y)\dd y=0,\qquad
    \int_{\R^d}\nabla^2 p_t(y)\dd y=0,\qquad
    \int_{\R^d}\nabla\partial_t p_t(y)\dd y=0,
    \]
    and applying the standard Gaussian bounds. For (iii) using $\partial_tP_tf=\frac12\Delta P_tf$ and (ii), we get
    \[
    |P_tf(x)-P_sf(x)|\lesssim\int_s^t[\nabla P_rf]_{\Lip}\dd r
    \lesssim[f]_{\Lip}s^{-1/2}(t-s).
    \]
\end{proof}

\section*{Proof of Lemma 2.3}
\begin{proof}
    Let
    \[
    M_t^n:=\sup_{r\in[0,t]}|X_r^n|, \qquad t\in[0,1].
    \]
    Since $b$ is Lipschitz continuous,
    \[
    |b(x)|\leq |b(0)|+[b]_{\mathrm{Lip}}|x|.
    \]
    By the definition of the Euler--Maruyama scheme and
    $\kappa_n(s)\leq s$, we have
    \[
    \begin{aligned}
    M_t^n
    &\leq |x_0|
    +\int_0^t |b(X_{\kappa_n(s)}^n)|\,\mathrm{d} s
    +\sup_{r\in[0,t]}|B_r^H| \\
    &\leq |x_0|+|b(0)|
    +[b]_{\mathrm{Lip}}\int_0^t M_s^n\,\mathrm{d} s
    +\sup_{r\in[0,1]}|B_r^H|.
    \end{aligned}
    \]
    Since $t\mapsto M_t^n$ is continuous, Gronwall's inequality yields
    \[
    M_1^n
    \leq
    e^{[b]_{\mathrm{Lip}}}
    \Big(
    |x_0|+|b(0)|
    +\sup_{r\in[0,1]}|B_r^H|
    \Big).
    \]
    By Fernique's theorem,
    \[
    \mathbb{E}\Big[
    \sup_{r\in[0,1]}|B_r^H|^q
    \Big]<\infty
    \]
    for every $q\geq1$. Therefore,
    \[
    \sup_{n\geq1}
    \Big\|
    \sup_{t\in[0,1]}|X_t^n|
    \Big\|_{L^q}
    <\infty.
    \]
    In particular,
    \[
    \sup_{n\geq1}\sup_{t\in[0,1]}
    \|X_t^n\|_{L^q}<\infty.
    \]

    Using the linear growth of $b$, the uniform moment estimate above and (i) in Lemma 2.1, we have
    \begin{align*}
        \norm{X_v^n-X_u^n}_{L^q} &=\norm{\int_u^v b(X_{\kn(r)}^n)\dd r+B_v^H-B_u^H}_{L^q} \\
        &\le \int_u^v\norm{b(X_{\kn(r)}^n)}_{L^q}\dd r+\norm{B_v^H-B_u^H}_{L^q} \\
        &\le C(v - u) + C(v-u)^H \lesssim (v-u)^H,
    \end{align*}
    where $C$ is independent of $n$.
\end{proof}

\section*{Proof of Lemma 2.4}
\begin{proof}
    Let $s':=\lceil ns\rceil/n$, the smallest grid point larger or equal to $s$. In the case of $\kn(t)\le s$ the LHS is zero. For $\kn(t)>s$, we use the elementary inequality
    \[
    \norm{Y-\E^s[Y]}_{L^q}\le2\norm{Y-Z}_{L^q},
    \]
    valid for every $\F_s$-measurable random variable $Z$. Although $s'=\lceil ns\rceil/n$ may be strictly larger than $s$, we have
    \[
    \ph_{s'}^n=\ph_s^n+(s'-s)b(X_{\kn(s)}^n),
    \]
    since the drift is frozen on $[s,s']$. Hence $\ph_{s'}^n$ is $\F_s$-measurable. Thus
    \begin{align*}
    \norm{\ph_{\kn(t)}^n-\E^s[\ph_{\kn(t)}^n]}_{L^q}
    &\le2\norm{\ph_{\kn(t)}^n-(\ph_{s'}^n+(\kn(t)-s')b(X_s^n))}_{L^q}\\
    &\le2\norm{\int_{s'}^{\kn(t)}(b(X_{\kn(r)}^n)-b(X_s^n))\dd r}_{L^q}\\
    &\le2[b]_{\Lip}\int_{s'}^{\kn(t)}\norm{X_{\kn(r)}^n-X_s^n}_{L^q}\dd r\\
    &\lesssim[b]_{\Lip}\int_{s'}^{\kn(t)}(\kn(r)-s)^H\dd r\\
    &\lesssim[b]_{\Lip}(t-s)^{1+H},
    \end{align*}
    where the fourth inequality follows from Lemma 2.3.
\end{proof}

\section*{Proof of Lemma 3.1}
\begin{proof}
    Set
    \[
    \rate:=\begin{cases}
    n^{-1}(t-s)^H,&H\in(1/2,1),\\
    n^{-(1/2+H)+\varepsilon}(t-s)^{1/2+\varepsilon},&H\in(0,1/2].
    \end{cases}
    \]
    We prove the claim by applying the shifted stochastic sewing lemma (Lemma 2.5). For $t\in[0,1]$, set
    \[
    \A_t:=\int_0^t\bigl\{b(B_r^H+\ph_{\kn(r)}^n)-b(B_{\kn(r)}^H+\ph_{\kn(r)}^n)\bigr\}\dd r.
    \]
    For $s,t\in\shift$, define
    \[
    A_{s,t}:=\E^{s_1}\left[\int_s^t\bigl\{b(B_r^H+\E^{s_1}[\psi_r])-b(B_{\kn(r)}^H+\E^{s_1}[\psi_r])\bigr\}\dd r\right],
    \]
    where $s_1:=s-(t-s)$ and $\psi_r:=\ph_{\kn(r)}^n$.

    We first verify the estimate corresponding to (1) in Lemma 2.5. Put $s':=\kn(s)+2/n$.
    \[
    \norm{A_{s,t}}_{L^p}\le\Norm{\int_s^{s'\wedge t}\cdots}_{L^p}
    +\Norm{\int_{s'\wedge t}^t\cdots}_{L^p}=:I_1+I_2.
    \]
    Since $|s'-s|\le2/n$, (i) in Lemma 2.1 gives
    \[
    I_1\lesssim[b]_{\Lip}n^{-H}(s'\wedge t-s)\lesssim[b]_{\Lip}\rate.
    \]
    If $t-s'>1/n$, then by (ii) in Lemma 2.1 and (i) and (iii) in Lemma 2.2,
    \begin{align*}
    I_2&=\Bigl\|\int_{s'}^t\Bigl\{P_{c^2(s_1,r)}b(\E^{s_1}[B_r^H]+\E^{s_1}[\psi_r])
    -P_{c^2(s_1,r)}b(\E^{s_1}[B_{\kn(r)}^H]+\E^{s_1}[\psi_r])\\
    &\qquad+P_{c^2(s_1,r)}b(\E^{s_1}[B_{\kn(r)}^H]+\E^{s_1}[\psi_r])
    -P_{c^2(s_1,\kn(r))}b(\E^{s_1}[B_{\kn(r)}^H]+\E^{s_1}[\psi_r])\Bigr\}\dd r\Bigr\|_{L^p}\\
    &\le\int_{s'}^t\Bigl\{[P_{c^2(s_1,r)}b]_{\Lip}\norm{\E^{s_1}[B_r^H-B_{\kn(r)}^H]}_{L^p}
    +\norm{P_{c^2(s_1,r)}b-P_{c^2(s_1,\kn(r))}b}_{\infty}\Bigr\}\dd r\\
    &\lesssim[b]_{\Lip}\int_{s'}^t\Bigl\{\norm{\E^{s_1}[B_r^H-B_{\kn(r)}^H]}_{L^p}
    +(\kn(r)-s_1)^{-H}\bigl((r-s_1)^{2H}-(\kn(r)-s_1)^{2H}\bigr)\Bigr\}\dd r\\
    &=:[b]_{\Lip}(I_{2,1}+I_{2,2}).
    \end{align*}
    Since
    \begin{equation}
    \kn(r)-s_1\ge\kn(s)-s+2/n+t-s\ge1/n+t-s\ge r-\kn(r),\label{eq:3}
    \end{equation}
    for any $r\in[s',t]$, applying (iii) in Lemma 2.1 yields
    \[
    \norm{\E^{s_1}[B_r^H-B_{\kn(r)}^H]}_{L^p}\lesssim(r-\kn(r))(r-s_1)^{H-1}.
    \]
    Therefore
    \[
    I_{2,1}\lesssim n^{-1}\int_{s'}^t(r-s_1)^{H-1}\dd r\lesssim n^{-1}(t-s)^H.
    \]
    By \eqref{eq:3} for any $r\in[s',t]$, we have
    \begin{align*}
    (r-s_1)^{2H}-(\kn(r)-s_1)^{2H}
    &\lesssim\begin{cases}
    n^{-1}(r-s_1)^{2H-1}&H\in(1/2,1)\\
    n^{-1}(\kn(r)-s_1)^{2H-1}&H\in(0,1/2]
    \end{cases}\\
    &\lesssim n^{-1}(t-s)^{2H-1}.
    \end{align*}
    Therefore, using $t-s\lesssim\kn(r)-s_1$,
    \[
    I_{2,2}\lesssim n^{-1}\int_{s'}^t(\kn(r)-s_1)^{-H}(t-s)^{2H-1}\dd r\le n^{-1}(t-s)^H.
    \]
    For $H\le1/2$, since $n^{-1} < t-s$ in the present case, the bound $n^{-1}(t-s)^H$ is dominated by $n^{-(1/2+H)+\varepsilon}(t-s)^{1/2+\varepsilon}$. Thus
    \[
    I_{2,1}+I_{2,2}\lesssim\rate.
    \]
    If $t-s'\le1/n$, the same argument as for $I_1$ yields the same bound. Therefore,
    \[
    \norm{A_{s,t}}_{L^p}\lesssim[b]_{\Lip}\rate.
    \]
    This verifies (1) in Lemma 2.5.

    We next verify the estimate corresponding to (2) in Lemma 2.5. Let $s,t\in\shift$, $u:=\frac{t+s}{2}$, and set $s_2:=s-(u-s)$, $s_3:=s$, $s_4:=u$, $s_5:=t$. Note that $0\le s_1\le s_2\le s_3\le s_4\le s_5\le t$.
    \begin{equation}
    \begin{aligned}
    \E^{s-(t-s)}[\delta A_{s,u,t}]&=\E^{s_1}[\delta A_{s_3,s_4,s_5}]\\
    &=\E^{s_1}\Biggl[\E^{s_2}\Bigl[\int_{s_3}^{s_4}\bigl\{b(B_r^H+\E^{s_1}[\psi_r])-b(B_{\kn(r)}^H+\E^{s_1}[\psi_r])\\
    &\qquad\qquad-b(B_r^H+\E^{s_2}[\psi_r])+b(B_{\kn(r)}^H+\E^{s_2}[\psi_r])\bigr\}\dd r\Bigr]\\
    &\qquad+\E^{s_3}\Bigl[\int_{s_4}^{s_5}\bigl\{b(B_r^H+\E^{s_1}[\psi_r])-b(B_{\kn(r)}^H+\E^{s_1}[\psi_r])\\
    &\qquad\qquad-b(B_r^H+\E^{s_3}[\psi_r])+b(B_{\kn(r)}^H+\E^{s_3}[\psi_r])\bigr\}\dd r\Bigr]\Biggr]\\
    &=:I_3+I_4.
    \end{aligned}\label{eq:4}
    \end{equation}
    We split the proof into the cases $t-s\ge4/n$ and $t-s<4/n$.

    \textbf{Case 1:} $t-s\ge4/n$. Note that if $r\in[s_3,s_4]$ then we have
    \begin{equation}
    \kn(r)-s_2\ge\kn(s)-s+(u-s)\ge-\frac1n+\frac{t-s}{2}\ge\frac{t-s}{4}\ge\frac1n.\label{eq:5}
    \end{equation}
    Therefore using (ii) in Lemma 2.1,
    \begin{equation}
    I_3=\E^{s_1}\left[\int_{s_3}^{s_4}\bigl\{I_{3,1}+I_{3,2}\bigr\}\dd r\right],\label{eq:6}
    \end{equation}
    where
    \begin{align*}
    I_{3,1}&:=P_{c^2(s_2,r)}b(\E^{s_2}[B_r^H]+\E^{s_1}[\psi_r])-P_{c^2(s_2,\kn(r))}b(\E^{s_2}[B_r^H]+\E^{s_1}[\psi_r])\\
    &\quad-P_{c^2(s_2,r)}b(\E^{s_2}[B_r^H]+\E^{s_2}[\psi_r])+P_{c^2(s_2,\kn(r))}b(\E^{s_2}[B_r^H]+\E^{s_2}[\psi_r]),\\
    I_{3,2}&:=P_{c^2(s_2,\kn(r))}b(\E^{s_2}[B_r^H]+\E^{s_1}[\psi_r])-P_{c^2(s_2,\kn(r))}b(\E^{s_2}[B_{\kn(r)}^H]+\E^{s_1}[\psi_r])\\
    &\quad-P_{c^2(s_2,\kn(r))}b(\E^{s_2}[B_r^H]+\E^{s_2}[\psi_r])+P_{c^2(s_2,\kn(r))}b(\E^{s_2}[B_{\kn(r)}^H]+\E^{s_2}[\psi_r]).
    \end{align*}
    By (iv) in Lemma 2.2, we have
    \begin{align*}
    |I_{3,1}|&\le[P_{c^2(s_2,r)}b-P_{c^2(s_2,\kn(r))}]_{\Lip}|\E^{s_1}[\psi_r]-\E^{s_2}[\psi_r]|\\
    &\lesssim[b]_{\Lip}(\kn(r)-s_2)^{-2H}\{(r-s_2)^{2H}-(\kn(r)-s_2)^{2H}\}|\E^{s_1}[\psi_r]-\E^{s_2}[\psi_r]|.
    \end{align*}
    By \eqref{eq:5}, for any $r\in[s_3,s_4]$, we have
    \[
    (r-s_2)^{2H}-(\kn(r)-s_2)^{2H}\lesssim n^{-1}(t-s)^{2H-1}.
    \]
    By Lemma 2.4, for any $q\ge2$ we have
    \begin{equation}
    \begin{aligned}
    \norm{\E^{s_1}[\psi_r]-\E^{s_2}[\psi_r]}_{L^q}
    &\le\norm{\E^{s_1}[\psi_r]-\psi_r}_{L^q}+\norm{\psi_r-\E^{s_2}[\psi_r]}_{L^q}\\
    &\le C[b]_{\Lip}(r-s_1)^{1+H},
    \end{aligned}\label{eq:7}
    \end{equation}
    for some constant $C>0$ independent of $n$. Combining the above estimates,
    \begin{equation}
    \begin{aligned}
    \norm{\E^{s_1}\left[\int_{s_3}^{s_4}I_{3,1}\dd r\right]}_{L^p}
    &\lesssim[b]_{\Lip}^2n^{-1}\int_{s_3}^{s_4}(t-s)^{-1}(r-s_1)^{1+H}\dd r\\
    &\lesssim[b]_{\Lip}^2n^{-1}(t-s)^{1+H}.
    \end{aligned}\label{eq:8}
    \end{equation}
    Using the elementary estimate
    \begin{align*}
    |f(x_1)-f(x_2)-f(x_3)+f(x_4)| \le [\nabla f]_{\Lip}|x_2-x_1||x_3-x_1|+[f]_{\Lip}|x_1-x_2-x_3+x_4|,
    \end{align*}
    valid for $f\in C^1$ with $f$ and $\nabla f$ Lipschitz continuous, we obtain
    \begin{align*}
    |I_{3,2}|&\le[\nabla P_{c^2(s_2,\kn(r))}b]_{\Lip}|\E^{s_2}[B_r^H]-\E^{s_2}[B_{\kn(r)}^H]||\E^{s_1}[\psi_r]-\E^{s_2}[\psi_r]|\\
    &\lesssim[b]_{\Lip}(\kn(r)-s_2)^{-H}|\E^{s_2}[B_r^H]-\E^{s_2}[B_{\kn(r)}^H]||\E^{s_1}[\psi_r]-\E^{s_2}[\psi_r]|.
    \end{align*}
    Since $r-\kn(r)\le1/n\le\kn(r)-s_2$ for any $r\in(s_3,s_4)$, applying (iii) in Lemma 2.1 yields
    \begin{equation}
    \norm{\E^{s_2}[B_r^H]-\E^{s_2}[B_{\kn(r)}^H]}_{L^{2p}}\lesssim(r-\kn(r))(r-s_2)^{H-1}.\label{eq:9}
    \end{equation}
    Combining \eqref{eq:5}, \eqref{eq:7} and \eqref{eq:9}, together with H\"older's inequality and the $L^p$-contraction property of conditional expectation,
    \begin{equation}
    \begin{aligned}
    &\norm{\E^{s_1}\left[\int_{s_3}^{s_4}I_{3,2}\dd r\right]}_{L^p}\\
    &\qquad\le C[b]_{\Lip}^2\int_{s_3}^{s_4}(t-s)^{-H}(r-\kn(r))(r-s_2)^{H-1}(r-s_1)^{1+H}\dd r\\
    &\qquad\lesssim[b]_{\Lip}^2n^{-1}(t-s)^{1+H},
    \end{aligned}\label{eq:10}
    \end{equation}
    where we used $r-s_2\ge(t-s)/2$. Combining \eqref{eq:6}, \eqref{eq:8} and \eqref{eq:10},
    \begin{equation}
    \norm{I_3}_{L^p}\lesssim[b]_{\Lip}^2n^{-1}(t-s)^{1+H}.\label{eq:11}
    \end{equation}
    The term $I_4$ is estimated in the same way as $I_3$, with $s_2$ replaced by $s_3$. Indeed, for $r\in[s_4,s_5]$ we have
    \[
    \kn(r)-s_3\ge\kn(u)-u+u-s\ge\frac{t-s}{4}\ge\frac1n,
    \]
    so that the same heat-kernel estimates apply. Hence
    \begin{equation}
    \norm{I_4}_{L^p}\lesssim[b]_{\Lip}^2n^{-1}(t-s)^{1+H}.\label{eq:12}
    \end{equation}
    \textbf{Case 2:} If $t-s<4/n$, then the Lipschitz continuity of $b$ and the same estimate as in \eqref{eq:7} give
    \[
    \norm{I_3}_{L^p}\lesssim[b]_{\Lip}^2(t-s)^{2+H}\lesssim[b]_{\Lip}^2n^{-1}(t-s)^{1+H},
    \]
    and the same bound holds for $I_4$.

    Combining \eqref{eq:4}, \eqref{eq:11}, \eqref{eq:12}, and the estimates in Case 2,
    \[
    \norm{\E^{s_1}[\delta A_{s,u,t}]}_{L^p}\lesssim[b]_{\Lip}^2n^{-1}(t-s)^{1+H},
    \]
    and (2) holds.

    It remains to prove (i) in Lemma 2.5. Let $(s,t)\in\Delta_{[0,1]}$, $s<t$ and let $m\in\mathbb N$. Define $t_i:=s+\frac{t-s}{m}i$, $i=0,1,\ldots,m$.
    \begin{align*}
    \norm{\A_t-\A_s-\sum_{i=1}^{m-1}A_{t_i,t_{i+1}}}_{L^1}
    &\le\norm{\A_{t_1}-\A_{t_0}}_{L^1}\\
    &\quad+\norm{\sum_{i=1}^{m-1}\bigl\{\A_{t_{i+1}}-\A_{t_i}-\E^{t_i}[\A_{t_{i+1}}-\A_{t_i}]\bigr\}}_{L^2}\\
    &\quad+\norm{\sum_{i=1}^{m-1}\bigl\{\E^{t_i}[\A_{t_{i+1}}-\A_{t_i}]-\E^{t_{i-1}}[\A_{t_{i+1}}-\A_{t_i}]\bigr\}}_{L^2}\\
    &\quad+\norm{\sum_{i=1}^{m-1}\bigl\{\E^{t_{i-1}}[\A_{t_{i+1}}-\A_{t_i}]-A_{t_i,t_{i+1}}\bigr\}}_{L^1}\\
    &=:H_1+H_2+H_3+H_4.
    \end{align*}
    By the Lipschitz continuity of $b$, we have
    \begin{equation}
    \begin{aligned}
    H_1&=\norm{\int_{t_0}^{t_1}\{b(B_r^H+\ph_{\kn(r)}^n)-b(B_{\kn(r)}^H+\ph_{\kn(r)}^n)\}\dd r}_{L^2}\\
    &\le\int_{t_0}^{t_1}\norm{b(B_r^H+\ph_{\kn(r)}^n)-b(B_{\kn(r)}^H+\ph_{\kn(r)}^n)}_{L^2}\dd r\\
    &\le[b]_{\Lip}\int_{t_0}^{t_1}\norm{B_r^H-B_{\kn(r)}^H}_{L^2}\dd r\lesssim t_1-t_0\lesssim m^{-1}.
    \end{aligned}\label{eq:13}
    \end{equation}
    Since the sequence $(\A_{t_{i+1}}-\A_{t_i}-\E^{t_i}[\A_{t_{i+1}}-\A_{t_i}])_{i=1,\ldots,m-1}$ is a martingale difference w.r.t.\ $(\F_{t_{i+1}})_{i=1,\ldots,m-1}$, by Burkholder--Davis--Gundy inequality,
    \begin{equation}
    \begin{aligned}
    H_2^2&\le\sum_{i=1}^{m-1}\norm{\A_{t_{i+1}}-\A_{t_i}-\E^{t_i}[\A_{t_{i+1}}-\A_{t_i}]}_{L^2}^2\\
    &\le4\sum_{i=1}^{m-1}\norm{\int_{t_i}^{t_{i+1}}\{b(B_r^H+\ph_{\kn(r)}^n)-b(B_{\kn(r)}^H+\ph_{\kn(r)}^n)\}\dd r}_{L^2}^2\lesssim m^{-1}.
    \end{aligned}\label{eq:14}
    \end{equation}
    Similarly, the sequence $(\E^{t_i}[\A_{t_{i+1}}-\A_{t_i}]-\E^{t_{i-1}}[\A_{t_{i+1}}-\A_{t_i}])_{i=1,\ldots,m-1}$ is a martingale difference w.r.t.\ $(\F_{t_i})_{i=1,\ldots,m-1}$, we get
    \begin{equation}
    H_3^2\lesssim\sum_{i=1}^{m-1}\norm{\E^{t_i}[\A_{t_{i+1}}-\A_{t_i}]-\E^{t_{i-1}}[\A_{t_{i+1}}-\A_{t_i}]}_{L^2}^2\lesssim m^{-1}.\label{eq:15}
    \end{equation}
    Using Lemma 2.4, we have
    \begin{align*}
    &\norm{\E^{t_{i-1}}[\A_{t_{i+1}}-\A_{t_i}]-A_{t_i,t_{i+1}}}_{L^1}\\
    &\quad\le\norm{\int_{t_i}^{t_{i+1}}\E^{t_{i-1}}[b(B_r^H+\ph_{\kn(r)}^n)-b(B_r^H+\E^{t_{i-1}}[\ph_{\kn(r)}^n])]\dd r}_{L^1}\\
    &\quad\phantom{\le}+\norm{\int_{t_i}^{t_{i+1}}\E^{t_{i-1}}[b(B_{\kn(r)}^H+\ph_{\kn(r)}^n)-b(B_{\kn(r)}^H+\E^{t_{i-1}}[\ph_{\kn(r)}^n])]\dd r}_{L^1}\\
    &\quad\le2[b]_{\Lip}\int_{t_i}^{t_{i+1}}\norm{\E^{t_{i-1}}[|\ph_{\kn(r)}^n-\E^{t_{i-1}}[\ph_{\kn(r)}^n]|]}_{L^1}\dd r\\
    &\quad\le C[b]_{\Lip}^2\int_{t_i}^{t_{i+1}}(r-t_{i-1})^{1+H}\dd r\\
    &\quad\lesssim(t_{i+1}-t_{i-1})^{2+H}=\left(\frac2m(t-s)\right)^{2+H}\lesssim m^{-(2+H)}.
    \end{align*}
    Therefore
    \begin{equation}
    H_4\le Cm^{-(1+H)}.\label{eq:16}
    \end{equation}
    Combining \eqref{eq:13}, \eqref{eq:14}, \eqref{eq:15} and \eqref{eq:16} we have
    \[
    \norm{\A_t-\A_s-\sum_{i=1}^{m-1}A_{t_i,t_{i+1}}}_{L^1}\to0\quad\text{as }m\to\infty,
    \]
    which implies (i) in Lemma 2.5. Applying Lemma 2.5, we have the desired conclusion.
\end{proof}
\end{document}